\documentclass[a4paper,11pt]{article}
\usepackage[top=2.5cm,bottom=2.5cm,left=2.2cm,right=2.2cm]{geometry}
\usepackage{amsfonts}
\usepackage{mathrsfs,amscd,amssymb,amsthm,amsmath,bm,graphicx,psfrag,subfigure,url,mathtools}
\usepackage{pict2e}
\usepackage{stfloats}
\usepackage{psfrag,amsmath}
\usepackage{tikz}
\usepackage{indentfirst}
\usepackage{hyperref}
\usepackage{bookmark}
\usepackage{enumerate}
\usepackage{latexsym,euscript,epic,eepic,color}
\usepackage{multirow}
\usepackage{multicol}
\usepackage{longtable}
\usepackage{adjustbox}
\usepackage[all]{xy}
\usepackage{setspace}
\usepackage{booktabs}
\usepackage{epstopdf}
\allowdisplaybreaks
\usepackage{authblk}
\usepackage{cite}

\makeatletter

\renewcommand{\@seccntformat}[1]{{\csname the#1\endcsname}{\normalsize .}\hspace{.5em}}
\makeatother

\usepackage{ifpdf}
\usepackage{indentfirst}

\def \[{\begin{equation}}
	\def \]{\end{equation}}

\newtheorem{thm}{Theorem}[section]

\newtheorem{defi}{Definition}
\newtheorem{claim}{Claim}
\newtheorem{lemma}[thm]{Lemma} 
\newtheorem{cor}[thm]{Corollary}

\newtheorem{fact}{Fact}

\newenvironment{wst}
{\setlength{\leftmargini}{1.5\parindent}
	\begin{itemize}
		\setlength{\itemsep}{-1.1mm}}
	{\end{itemize}}
\begin{document}
	\baselineskip=0.23in
	\title{\bf  
        Extending two results on hamiltonian graphs involving the bipartite-hole-number\thanks{
			{\it Email addresses}: chengkunmath@163.com (K. Cheng), tyr2290@163.com (Y. Tang).}}
	\author{Kun Cheng}
	\author{Yurui Tang\thanks{Corresponding author}}
	\affil{Department of Mathematics, East China Normal University, Shanghai, 200241, China}
	\date{\today}
	\maketitle
	\begin{abstract}
		The bipartite-hole-number of a graph $G$, denoted by $\widetilde{\alpha}(G)$, is the minimum number $k$ such that there exist positive integers $s$ and $t$ with $s+t=k+1$ with the property that for any two disjoint sets $A,B\subseteq V(G)$ with $|A|=s$ and $|B|=t$, there is an edge between $A$ and $B$. In this paper, we first prove that any $2$-connected graph $G$ satisfying  $d_G(x)+d_G(y)\ge 2\widetilde{\alpha}(G)-2$ for every  pair of non-adjacent vertices $x,y$  is hamiltonian except for a special family of graphs, thereby extending results of Li and Liu (2025), and Ellingham, Huang and  Wei (2025). 
		We then establish a stability version of a theorem by McDiarmid and Yolov (2017): every graph whose minimum degree is at least its bipartite-hole-number minus one is hamiltonian except for a special family of graphs. 
		\vspace{0.2cm}
		
	\end{abstract} 
	
	\section{\normalsize Introduction}
	We consider finite simple graphs and use standard terminology and notation from \cite{bondy1} and \cite{West}.    
	Let $G$ be a graph with vertex set $V(G)$ and edge set $E(G)$. Then $n(G):=|V(G)|$ and $e(G):=|E(G)|$ are called the {\it order} and {\it size} of $G$, respectively.  
	For $v\in V(G)$, let $N_G(v)$ and $d_G(v)$ be the neighborhood and the degree of $v$ in $G$, respectively. 
	Let $\delta(G)=\min\{d_G(v): v\in V(G)\}$ be the minimum degree of $G$. 
	We write $P_n$, $C_n$ and $K_n$ for the path, the cycle and the complete graph of order $n$, respectively.  
	For two graphs $G$ and $H,$ $G\vee H$ denotes the {\it join} of $G$ and $H,$ which is obtained from the disjoint union $G+H$ by adding edges joining every vertex of $G$ to every vertex of $H.$  
	For graphs we will use equality up to isomorphism, so $G = H$ means that $G$ and $H$ are isomorphic.

	For two distinct vertices $x$ and $y$ in $G$, an {\em $(x, y)$-path}  is a path whose endpoints are $x$ and $y$.  
	A Hamilton cycle (resp., path) in $G$ is a cycle (resp., path) containing every vertex of $G$.  
	$G$ is {\em hamiltonian} (resp., {\em traceable}) if it contains a Hamilton cycle (resp., path).  
	A graph is called  {\em  hamiltonian-connected} if between any 
	two distinct vertices there is a Hamilton path. 
	The classic  Dirac theorem~\cite{Dirac1952} from 1952 states that 
	if  $G$ is a graph of order $n\ge 3$ and $\delta(G)\geq n/2$, then $G$ is hamiltonian. 
    Let $$\sigma_2(G):=\min\{d_G(x)+d_G(y):x,y\in V(G) \text{~~and~~} xy\not\in E(G)\}.$$ 
	Ore~\cite{Ore1960}  weakened Dirac's condition and showed that a graph $G$ of order $n$ is hamiltonian if $\sigma_2(G)\ge n.$  
    The length of a shortest  $(x,y)$-path in a graph $G$ is called the {\em distance} between $x$ and $y$, and denoted $d_G(x,y).$ 
    In 1984, Fan~\cite{Fan} introduced a condition for pairs of vertices at distance $2$, 
	and proved that 
	if $G$ is a $2$-connected graph of order $n$ such that $\max\{d_G(x),d_G(y)\}\ge n / 2$ for every pair of vertices $x,y$ with $d_G(x,y)=2$, 
	then $G$ is hamiltonian. 
	A well-known result of Chv\'{a}tal and Erd\H{o}s~\cite{Chvatal1972} from 1972 states that a graph with connectivity not less than its independence number is hamiltonian.
	For more  results in this field, 
	we refer the reader to two nice surveys \cite{Gould2014,Li2013}.
	
    For a graph $G$ and two disjoint vertex sets $X,Y\subseteq V(G)$, we use $E(X,Y)$ to denote the set of edges with one end in $X$ and the other in $Y$. Let $e(X,Y):=|E(X,Y)|.$  
	In 2017, McDiarmid and Yolov~\cite{Mcdiarmid2017} introduced the concept of bipartite hole in the study the  Hamilton cycles.
	\begin{defi}
		An $(s,t)$-bipartite-hole in a graph $G$ consists of two disjoint sets of vertices, $S$ and $T$, with $|S| = s$ and $|T| = t$, such that $E(S, T)= \emptyset$. 
		The {\it bipartite-hole-number} of a graph $G$, denoted by $\widetilde{\alpha}(G)$, is the minimum number $k$ such that there exist positive integers $s$ and $t$ with $s + t = k + 1$, and such that $G$ does not contain an $(s,t)$-bipartite-hole.
	\end{defi}
	
	An equivalent definition of $\widetilde{\alpha}(G)$ is the maximum integer $r$ such that $G$ contains an $(s,t)$-bipartite-hole for every pair of nonnegative integers $s$ and $t$ with $s+t=r$. 
	McDiarmid and Yolov \cite{Mcdiarmid2017} provided a sufficient condition for hamiltonicity  in terms of the minimum degree and the bipartite-hole-number of a graph.
	\begin{thm}[McDiarmid-Yolov \cite{Mcdiarmid2017}]\label{thmMc}
		Let $G$ be a graph of order at least three.  If $\delta(G)\ge \widetilde{\alpha}(G)$, then $G$ is hamiltonian.
	\end{thm}
	A graph $G$ with $\delta(G)\ge n/2$ contains no $(1, \lfloor n/2\rfloor)$-bipartite-hole, hence 
	$\delta(G) \ge n/2 \ge \widetilde{\alpha}(G),$ 
	so Theorem~\ref{thmMc} generalizes the  Dirac theorem.  
	
	In 2024, Zhou, Broersma, Wang and Lu~\cite{Zhou2024} raised the minimum degree bound in Theorem~\ref{thmMc} and showed that a graph $G$ is hamiltonian-connected if $\delta(G)\ge \widetilde{\alpha}(G)+1$. 
	A graph of order $n$ is {\em pancyclic} if it contains one cycle of each length 
	$l$, $3\le l\le n.$ 
	  Dragani\'{c},  Correia and Sudakov~\cite{Draganic2024}  showed   that a graph $G$ is pancyclic if $\delta(G)\ge \widetilde{\alpha}(G)$,  unless $n$ is even and $G$ is the balanced complete bipartite graph $K_{\frac{n}{2}, \frac{n}{2}}$. This  result generalizes both  Theorem~\ref{thmMc} and the classical Bondy pancyclicity theorem~\cite{bondy71}. 
      Li and Liu~\cite{Li-Liu} proved that a 2-connected graph $G$ is hamiltonian if $\sigma_2(G)\ge 2\widetilde{\alpha}(G).$ 
    Recently, Ellingham, Huang and  Wei~\cite[Theorem 1.4]{Ellingham} weakened this condition, showing that $\sigma_2(G)\ge 2\widetilde{\alpha}(G)-1$ is already sufficient. For more recent works on the bipartite-hole-number, we refer the readers to \cite{Chen2022,Han2024,LLT,LiuHc}.  
    \begin{thm}[Ellingham-Huang-Wei~\cite{Ellingham}]\label{thmEl} 
    Let $G$ be a $2$-connected graph of order at least three.  If $\sigma_2(G)\ge 2\widetilde{\alpha}(G)-1$, then $G$ is hamiltonian.
    \end{thm}    
    
	In this paper, 
	we first extend the results of Li and Liu~\cite{Li-Liu}  and  Ellingham, Huang and  Wei~\cite{Ellingham} as follows.
	\begin{thm}\label{thmctOre}
		Let $G$ be a $2$-connected graph of order $n\ge 3$.  
		If $\sigma_2(G)\ge 2\widetilde{\alpha}(G)-2$, 
		then $G$ is hamiltonian unless 
		$G= G_{ \frac{n-1}{2}}\vee \frac{n+1}{2} K_1$, where $n\ge 5$ is odd and  $G_{ \frac{n-1}{2}}$ is an arbitrary graph of order $(n-1)/2$. 
	\end{thm}
	
	By applying Theorem~\ref{thmctOre}, we obtain a stability version of Theorem~\ref{thmMc}.
	\begin{thm}\label{thmct1}
		Let $G$ be a graph of order $n\ge 3$.  
		If $\delta(G)\ge \widetilde{\alpha}(G)-1$, 
		then $G$ is hamiltonian unless 
		$G\in\{ G_{ \frac{n-1}{2}}\vee \frac{n+1}{2} K_1, K_1\vee 2K_{ \frac{n-1}{2}} \}$, where $n$ is odd and $G_{ \frac{n-1}{2}}$ is an arbitrary graph of order $(n-1)/2$. 
	\end{thm}

	The remaining sections are organized as follows: 
	In Section~\ref{s2}, we introduce some notations. 
	In Section~\ref{s3}, we  prove Theorems~\ref{thmctOre} and \ref{thmct1}. 
	Some concluding remarks are given in the last section.

	\section{\normalsize Preliminaries}\label{s2} 
	Let $G$ be a graph.  
	For $u, v\in V(G)$, we denote by $u\sim v$ if $u$ and $v$ are adjacent, and $u\not\sim v$ otherwise.  
	A set of vertices  is {\em independent} if no two of its 
	elements are adjacent. 
	For $S\subseteq V(G)$, we write $N_S(v)$ for $N_G(v)\cap S$ and let $d_S(v):=|N_S(v)|$. 
	Similarly, for a subgraph $H$ of $G$, let $N_H(v)=N_G(v)\cap V(H)$ 
	and let $d_H(v):=|N_H(v)|$. 
	
	Let $P=v_1v_2\dots v_l$ be a path in a graph $G$.
	We set 
	\begin{align*}
		\text{
			$v_i^+:=v_{i+1}$,~~ 
			$v_j^-:=v_{j-1}$,~~  
			$v_i^{+p}:=v_{i+p}$~~and~~
			$v_j^{-q}:=v_{j-q}$,
	}\end{align*}
	respectively.  
    For any vertex subset $S\subseteq V(P)$, 
	we write 
    \begin{align*}
		S^{+p}:=\{v_i^{+p} : v_i\in S\setminus \{v_{l-p+1},\ldots,v_l\}\}\text{~~and~~}
		S^{-q}:=\{v_i^{-q} : v_i\in S\setminus \{v_1,\ldots,v_{q}\}\}.
	\end{align*}  
    Set $S^{+}=S^{+1}$ and $S^{-}=S^{-1}$ for short. 
	If $a<b,$
	we set 
	\begin{align*}
		v_a \overrightarrow{P} v_b:=v_a v_{a+1}\ldots  v_b \text{~~and~~} 
		v_b \overleftarrow{P} v_a:=v_b v_{b-1}\ldots  v_a.
	\end{align*}   
    Let $[v_a,v_b]=V(v_a \overrightarrow{P} v_b).$ 
    
	By joining a new vertex to a graph $G$, Zhou, Broersma, Wang and Lu~\cite{Zhou2024} obtained the following result from Theorem~\ref{thmMc}. 
	\begin{lemma}[Zhou-Broersma-Wang-Lu~\cite{Zhou2024}]\label{lemZ}
		A graph $G$ of order at least three is traceable if $\delta(G)\ge \widetilde{\alpha}(G)-1.$
	\end{lemma}

	\section{\normalsize The proofs}\label{s3}
	
	\begin{proof}[\bf Proof of Theorem~\ref{thmctOre}] 
    Let $G$ be a non-hamiltonian graph satisfying the conditions of Theorem~\ref{thmctOre}.  We are to show that $G=G_{ \frac{n-1}{2}}\vee \frac{n+1}{2} K_1,$ where $n\ge 5$ and $G_{ \frac{n-1}{2}}$ is an arbitrary graph of order $(n-1)/2$.  

    By the definition of the bipartite-hole-number, 
    we have 
    $$\sigma_2(G\vee K_1)= \sigma_2(G)+2\ge 2\widetilde{\alpha}(G)=2\widetilde{\alpha}(G\vee K_1).$$ 
    Then by Theorem~\ref{thmEl}, 
    $G\vee K_1$ is hamiltonian. 
    Thus $G$ is traceable. 
		  Let  $P=v_1v_2\dots v_n$ be a Hamilton path of $G$  such that $d_G(v_1) \le d_G(v_n)$, $d_G(v_1)$ is as  large as possible, 
         and subject to this, $d_G(v_n)$ is as large as possible. 
	
		By  definition, there exist two 
		positive integers $s$ and $t$ such that $s+t=\widetilde{\alpha}(G) +1$ 
		and $G$ contains no $(s,t)$-bipartite-hole.  
		Without loss of generality, we assume that $1\le s\le t.$ 
		Since $G$ has no Hamilton cycle, $G\neq K_n$, and hence $t\ge 2$ and $\widetilde{\alpha}(G) \ge 2.$ 

 \begin{claim}\label{c0}
    $d_G(v_1)= \widetilde{\alpha}(G)-1=s+t-2.$ 
 \end{claim}       
\begin{proof}[Proof of Claim~\ref{c0}] 
To the contrary, we first suppose that $d_G(v_1)\ge \widetilde{\alpha}(G)=s+t-1\ge s+1$. Let $$k=\min\left\{i : |N(v_1)\cap [v_1,v_i]|=s\right\}.$$ 
Since $G$ is non-hamiltonian, we have 
\begin{align}\label{e21}
    E\left(  (N_P(v_1)\cap [v_1,v_k])^-,(N_P(v_n)\cap [v_k,v_n])^+ \right)=\emptyset,
\end{align}
and
\begin{align}\label{e22}
E\left( (N_P(v_1)\cap [v_k,v_n])^+, (N_P(v_n)\cap [v_1,v_k))^+ \right)=\emptyset. 
\end{align} 
Since $|N_P(v_1)\cap [v_1,v_k]|=s$, 
we have $|N_P(v_1)\cap [v_k,v_n]| \ge \widetilde{\alpha}(G)-s+1=t$. Combining  this with \eqref{e21}-\eqref{e22} and  the fact that $G$ contains no $(s,t)$-bipartite-hole, we have 
\begin{align*}
    |N_P(v_n)\cap [v_k,v_n]|\le t-1 \text{~~and~~} |N_P(v_n)\cap [v_1,v_k)|\le s-1.  
\end{align*}
Thus, $d_G(v_n)=|(N_P(v_n)\cap [v_1,v_k))\cup(N_P(v_n)\cap [v_k,v_n])|\le s+t-2=\widetilde{\alpha}(G)-1,$ which contradicts the fact that $d_G(v_1) \le d_G(v_n)$.

Now we  suppose to the contrary that  $d_G(v_1)\le \widetilde{\alpha}(G)-2=s+t-3$. Let $h:=\max\{i: v_1\sim v_i\}$. 
         Since $G$ is 2-connected and non-hamiltonian, we have 
         $3\le h\le n-1.$ 
         If $h=3$, then $N_G(v_1)=\{v_2,v_{3}\}$. 
         If $h\ge 4,$ 
         then $G$ contains a Hamilton path $v_{h-1}\overleftarrow{P}v_1v_h\overrightarrow{P}v_n$. 
         By the choice of $P$, 
         we have $d_G(v_{h-1})\le d_G(v_1)\le \widetilde{\alpha}(G)-2$. 
        Combining with $\sigma_2(G)\ge 2\widetilde{\alpha}(G)-2$, we have  $v_{h-1}\sim v_1$. 
         If $h=4$ we stop, 
         if not, 
         then we repeat the process above and thus we obtain that 
         \begin{align}\label{et0}
             \text{  
             $d_G(v_{a})\le d_G(v_1)\le \widetilde{\alpha}(G)-2$ for each $2\le a\le h-1,$~~and~~
         $N_G(v_1)=\{v_2,\ldots,v_{h}\}$. 
             }
         \end{align}
         Since $G$ is $2$-connected, there exists $v_iv_j\in E(G)$, where $2\le i\le h-1$ and $h+1\le j\le n$. This implies that $G$ contains a Hamilton path $v_{j-1}\overleftarrow{P}v_{i+1}v_1\overrightarrow{P}v_iv_j\overrightarrow{P}v_n$. Then 
         \begin{align}\label{et4}
             \text{  
         $v_{j-1}\not\sim v_n$~~and~~$d_G(v_{j-1})\le d_G(v_1)\le \widetilde{\alpha}(G)-2$. 
         }
         \end{align}
         Combining this with \eqref{et0} and $\sigma_2(G)\ge 2\widetilde{\alpha}(G)-2$, we have  $\{v_1,\ldots,v_{h-1}\}\subseteq N_G(v_{j-1})$ and hence $j=h+1$. Together with $v_{h+1}\sim v_h$, we have  $d_G(v_{h})\ge d_G(v_1)+1$, which contradicts \eqref{et4}.  
         This proves Claim~\ref{c0}.
	\end{proof}
	From Claim~\ref{c0}, we have $d_G(v_1)= s+t-2\ge s$. Then we let $k:=\min\left\{i:|N(v_1)\cap [v_1,v_i]|=s\right\}$.
		Thus $v_1\sim v_k$.  
		We write 
		\begin{align*}
			A_1:=N_P(v_n)\cap [v_1,v_k) \text{~~and~~} A_2:=N_P(v_n)\cap [v_k,v_n].  
		\end{align*}
		Since $G$ is non-hamiltonian, we have 
		\begin{align}\label{3.2}
			E\left(  (N_P(v_1)\cap [v_1,v_k])^-,A_2^+ \right)=\emptyset,
		\end{align} 
		and 
		\begin{align}\label{3.3}
			E\left( (N_P(v_1)\cap [v_k,v_n])^+\cup\{v_1\},A_1^+ \right)=\emptyset.
		\end{align} 
		Note that $|(N_P(v_1)\cap [v_1,v_k])^-|=|N_P(v_1)\cap [v_1,v_k]|=s$. 
		Combining this with the fact that $G$ contains no $(s,t)$-bipartite-hole  and \eqref{3.2}, 
		we have 
		\begin{align}\label{eq2}
			|A_2|=|A_2^+|\le t-1.  
		\end{align}
		Together with $d_G(v_n)\ge d(v_1) = s+t-2$ gives 
		\begin{align}\label{eq1}
			|A_1^+|=|A_1|\ge s-1. 
		\end{align}
		Note that  
		$|(N_P(v_1)\cap [v_k,v_n])^+\cup\{v_1\}|
		=|N_P(v_1)\cap [v_k,v_n]|+1
		= t$. 
		Combining this with \eqref{3.3} and \eqref{eq1} gives 
		\begin{align}\label{eq3}
			|A_1^+|=|A_1|= s-1.
		\end{align}
		Now by \eqref{eq2}, \eqref{eq3} and 
		$|A_1|+|A_2|=d_G(v_n)\ge d_G(v_1)=s+t-2$, we have 
		\begin{align}\label{eq4}
			\text{
				$|A_2|= t-1$~~and~~
				$d_G(v_n)=s+t-2.$ 
			}
		\end{align}
		Thus $d_G(v_n)\ge  s$.  
		We let $k':=\max\{j:|N_P(v_n)\cap [v_j,v_n]=s\}$. 
		Then $v_{k'}\sim v_n$. 
		Write 
		\begin{align*}
			B_1:=N_P(v_1)\cap [v_1,v_{k'}] \text{~~and~~} B_2:=N_P(v_1)\cap (v_{k'},v_n].
		\end{align*} 
		Since $G$ is non-hamiltonian, we have 
		\begin{align}\label{3.2'}
			E\left((N_P(v_n)\cap [v_{k'},v_n])^+,B_1^-\right)=\emptyset,
		\end{align}
		and 
		\begin{align*}
			E\left((N_P(v_n)\cap [v_1,v_{k'}])^-\cup\{v_n\},B_2^-\right)=\emptyset. 
		\end{align*} 
		Similarly, we have 
		\begin{align}\label{34}
			\text{
				$|B_2|=s-1$~~and~~$|B_1|=t-1$.
			}
		\end{align}
		
		 Let $\ell:=s+t-2$. Based on \eqref{eq3} and \eqref{eq4}, we relabel  $N_P(v_n)$ as $\{u_1,u_2,\ldots,u_\ell\}$ according to their order along  $\overrightarrow{P}$. 
		Hence 
		\begin{align*}
			A_2=N_P(v_n)\cap [v_k,v_n]=\{u_s,\ldots,u_\ell\}~~\text{and}~~u_{\ell}=v_{n-1}.  
		\end{align*}        
		Similarly, we relabel $N_P(v_1)$ as $\{w_1,w_2,\dots, w_\ell\}$ according to their order along  $\overrightarrow{P}$.  Hence 
		\begin{align*}
			B_1=N_P(v_1)\cap [v_1,v_{k'}]=\{w_1,\dots,w_{t-1}\}~~\text{and}~~w_{1}=v_{2}.
		\end{align*} 
		
		\begin{claim}\label{special} 
			Assume that $t\ge 3$.
			\begin{wst}
				\item[{\rm (i)}] 
				$d_G(u_a^+)=s+t-2$ and $N_P(u_a^+)\cap [v_k,v_n]=\{u_s,\ldots, u_a,u_a^{+2},\ldots,u_{\ell-1}^{+2}\},$  for   $s \le a\le \ell-1$; 
				\item[{\rm (ii)}] 
				$d_G(w_b^-)=s+t-2$ and $N_P(w_b^-)\cap [v_1,v_{k'}]=\{w_2^{-2},\ldots, w_b^{-2},w_b,\ldots,w_{t-1}\},$ for   $2\le b\le t-1$. 
			\end{wst}
		\end{claim}
		\begin{proof}[Proof of Claim~\ref{special}]  
			(i) 
			For  $s \le a\le \ell-1$,    
			we consider the path $v_1\overrightarrow{P}u_a v_n\overleftarrow{P}u_a^+$. 
            Then $v_1\not\sim u_a^+.$ 
			Similar to \eqref{3.2}, we have 
			\begin{align}\label{3.5}
				E\left(  \left( N_P(v_1)\cap [v_1,v_k] \right)^-, \left( N_P(u_a^+)\cap [v_k,u_a^-] \right)^+ \cup \left( N_P(u_a^+)\cap[u_a^+,v_n] \right)^- \cup\{v_n\} \right)=\emptyset. 
			\end{align}
			By \eqref{3.2}, \eqref{eq4}, \eqref{3.5} and the fact 
			that $G$ contains no $(s,t)$-bipartite-hole, 
			we have 
			\begin{align}\label{37} 
				\left( N_P(u_a^+)\cap [v_k,u_a^-] \right)^+ \cup \left( N_P(u_a^+)\cap[u_a^+,v_n] \right)^- \cup\{v_n\}
				\subseteq  A_2^+.
			\end{align} 
			Thus 
			\begin{align*} 
				N_P(u_a^+)\cap [v_k,u_a^-] \subseteq  A_2\cap [v_k,u_a^-]=N_P(v_n)\cap [v_k,u_a^-]~~\text{and}~~N_P(u_a^+)\cap[u_a^+,v_n]\subseteq  A_2^{+2}\cap[u_a^+,v_n].
			\end{align*} 
			Therefore, 
			\begin{align} 
				&N_P(u_a^+)\cap[v_k,u_a]\subseteq N_P(v_n)\cap [v_k,u_a]=\{u_s,\ldots, u_{a}\}~~\text{and}~~\label{13}\\ &N_P(u_a^+)\cap[u_a^+,v_n]\subseteq  N_P^{+2}(v_n)\cap[u_a^+,v_n] =\{u_a^{+2},\ldots, u_{\ell-1}^{+2}\}.\label{12}
			\end{align} 
			Based on  $\left( N_P(u_a^+)\cap [v_k,u_a^-] \right)^+ \cap \left( N_P(u_a^+)\cap[u_a^+,v_n] \right)^-=\emptyset$, \eqref{37} and $|A_2^+|=|A_2|=t-1$, we have 
			\begin{align} 
				\left|N_P(u_a^+)\cap [v_k,v_n] \right|&=\left| N_P(u_a^+)\cap [v_k,u_a^-] \right|+ \left| N_P(u_a^+)\cap[u_a^+,n] \right| +1\notag
				\\&=\left| \left( N_P(u_a^+)\cap [v_k,u_a^-] \right)^+ \right| +\left| \left( N_P(u_a^+)\cap[u_a^+,n] \right)^-\right|+1\le t-1.\label{eq12} 
			\end{align} 
			Next, we show that 
            \begin{align}\label{e8}
            |N_P(u_a^+)\cap [v_1,v_{k-1}]|\le s-1. 
            \end{align}
			If $u_a\not\sim v_1,$ then 
			$$E\left( \left( N_P(v_1)\cap [v_k,u_a^-] \right)^+ \cup \left( N_P(v_1)\cap[u_a^{+},v_n] \right)^-\cup\{v_1\},\left( N_P(u_a^+)\cap [v_1,v_{k-1}] \right)^+ \right)=\emptyset,$$
			where $\left| \left( N_P(v_1)\cap [v_k,u_a^-] \right)^+ \cup \left( N_P(v_1)\cap[u_a^{+},v_n] \right)^-\cup\{v_1\}\right|=\left|N_P(v_1)\cap [v_k,v_n]\right|+1=t.$  Combining this with the fact that $G$ contains no $(s,t)$-bipartite-hole gives \eqref{e8}, as desired. 
			Now we assume that $u_a\sim v_1.$ Since 
			$G$ is non-hamiltonian, we have $E\left( v_n,\left( N_P(u_a^+)\cap [v_1,v_{k-1}] \right)^+ \right)=\emptyset.$ 
			Therefore, 
			$$E \left( \left( N_P(v_1)\cap [v_k,u_a^-] \right)^+ \cup \left( N_P(v_1)\cap[u_a^{+},v_n] \right)^-\cup\{v_1,v_n\},\left( N_P(u_a^+)\cap [v_1,v_{k-1}] \right)^+ \right)=\emptyset,$$ 
			where $\left| \left( N_P(v_1)\cap [v_k,u_a^-] \right)^+ \cup \left( N_P(v_1)\cap[u_a^{+},v_n] \right)^-\cup\{v_1,v_n\} \right|=|N_P(v_1)\cap [v_k,v_n]|-1+2=t.$ 
			Together with the fact that $G$ contains no $(s,t)$-bipartite-hole, we get \eqref{e8}, as desired.  
			
			In view of \eqref{eq12}-\eqref{e8}, we have  $d_G(u_a^+)\le s+t-2=\widetilde{\alpha}(G)-1.$ 
            Recall that $v_1\not\sim u_{a}^+$. 
            Thus $$d_G(v_1)+d_G(u_a^+)\ge 2\widetilde{\alpha}(G)-2.$$ 
            Therefore, $d_G(u_a^+)=\widetilde{\alpha}(G)-1=s+t-2.$ 
			Hence every inequality in \eqref{eq12}-\eqref{e8} becomes
equality.  
			Now combining this with  \eqref{13} and \eqref{12},  
			we infer that 
			every inequality in \eqref{13}-\eqref{12} becomes equality. 
			Hence   
			$N_P(u_a^+)\cap [v_k,v_n]=\{u_s,\ldots, u_a,u_a^{+2},\ldots,u_{\ell-1}^{+2}\}.$ This proves (i).  
			 A similar argument yields a proof of (ii).  This proves Claim~\ref{special}.
		\end{proof}

		\begin{claim}\label{inden}		
			$A_2^+$ (resp., $B_1^-$) is an independent set of $G$. 
		\end{claim}
		\begin{proof}[Proof of Claim~\ref{inden}] 
			Suppose to the contrary that $A_2^+$ is not an independent set of $G$. Then $|A_2|=t-1\ge 2$. 
            Assume that $u_{q_1},u_{q_2}\in A_2$ with $q_1<q_2$ and  $u_{q_1}^+\sim u_{q_2}^+.$ 
			We assert that 
             \begin{align}\label{et1}
                 \text{
                 $u_s=v_{n-t+1}$~~and~~$V\left( u_s \overrightarrow{P} v_{n-1}\right)\subseteq N_G(v_n).$ 
                 }
             \end{align} 
If $|A_2|=2$, then $A_2=\{u_{\ell-1}, u_{\ell}\}=\{u_{q_1},u_{q_2}\}$ and $s=\ell-1.$    Combining with $|A_2|=2$ and $u_{q_1}^+\sim u_{q_2}^+=v_n$, we have  $u_{q_1}=u_{\ell-1}=v_{n-2}$, and thus \eqref{et1} holds. 

Next we assume that  $|A_2|\ge 3.$
We show that there exist two integers 
			$s\le i_1<i_2\le \ell$ such that 
			$\left|u_{i_1}\overrightarrow{P} u_{i_2}  \right|\ge 3$ and 
			$u_{i_1}^+\sim u_{i_2}^+.$ 
                If $|u_{q_1}\overrightarrow{P}u_{q_2}|\ge 3$, then we write $q_1=i_1$ and $q_2=i_2$, and we are done. 
                Now assume that $|u_{q_1}\overrightarrow{P}u_{q_2}|=2$. Since $|A_2|\ge 3,$ there exists a neighbor $u_{q_3}$ of $v_n$. If $q_3<q_1,$ by Claim~\ref{special}, substituting $a=q_3$, we have $u_{q_3}^+\sim u_{q_1}^{+2}=u_{q_2}^+.$  
                Then we write $q_3=i_1$ and $q_2=i_2$ and we are done. 
                Now $|u_{i_1}\overrightarrow{P}u_{i_2}|\ge 3$ and $u_{i_1}^+\sim u_{i_2}^+.$ If $q_3>q_2,$ by Claim~\ref{special}, substituting $a=q_3$, we have $u_{q_3}^+\sim u_{q_2}=u_{q_1}^+$. Then we write $q_3=i_1$ and $q_1=i_2,$ and we are done.

			We assert that 
			$V\left( u_{i_1} \overrightarrow{P} u_{i_2}\right)\subseteq N_G(v_n).$
			Substituting $a= i_1$ in Claim~\ref{special}(i), we obtain that there exists an integer $i_3$ such that $u_{i_2}^+=u_{i_3}^{+2}$.  
			Hence, $u_{i_2}^-\sim v_n.$ 
			If $\left|u_{i_1}\overrightarrow{P} u_{i_2}  \right|=3$, we are done. 
			If $\left|u_{i_1}\overrightarrow{P} u_{i_2}  \right|\ge 4$, 
			then  substituting  $a=i_2$ in Claim~\ref{special}(i) gives that there exists an integer $i_4$ such that 
			$u_{i_4}=u_{i_1}^+$. 
			If $\left|u_{i_1}\overrightarrow{P} u_{i_2}  \right|= 4$, 
			we are done. 
			If $\left|u_{i_1}\overrightarrow{P} u_{i_2}  \right|\ge 5$, 
			then  substituting  $a=i_4$ in Claim~\ref{special}(i) gives that 
			$u_{i_4}^+ \sim u_{i_3}^{+2}=u_{i_2}^+.$ 
			Now substituting  $a=i_2$ in Claim~\ref{special}(i) gives that    
			there exists an integer $i_5$ such that 
			$u_{i_5}=u_{i_4}^+$. 
			Hence 
			$u_{i_4}^+\sim v_n.$ 
			If $\left|u_{i_1}\overrightarrow{P} u_{i_2}  \right|= 5$, 
			we are done. 
			If $\left|u_{i_1}\overrightarrow{P} u_{i_2}  \right|\ge 6$, then repeating the above argument  successively finitely
			many times, we obtain that $V\left( u_{i_1} \overrightarrow{P} u_{i_2}\right)\subseteq N_G(v_n).$ 
			
			We then show that $V\left( u_{i_2} \overrightarrow{P} v_{n-1}\right)\subseteq N_G(v_n).$ If $i_2=\ell$, we are done. Thus we assume that 
			$i_2\le \ell-1.$ Note that 
			$u_{\ell}^{+}=v_n\sim u_{i_2}=u_{i_3}^{+}.$ 
			Substituting  $a=i_3$ in Claim~\ref{special}(i) gives that 
			$u_{\ell-1}^{+2}=v_n$, and thus $u_{\ell-1}=v_{n-2}$. 
			If $i_2=\ell-1$, we are done. 
			If $i_2\le \ell-2,$ then 
			substituting  $a=\ell-1$ in Claim~\ref{special}(i) gives that 
			$u_{\ell-1}^+\sim u_{i_2}=u_{i_3}^+.$ 
			Now substituting  $a=i_3$ in Claim~\ref{special}(i) gives that 
			 $u_{\ell-1}^+=u_{\ell-2}^{+2}.$     
        If $i_2= \ell-2,$ we are done. 
         If $i_2\le \ell-3,$ 
        then repeating the above argument  successively finitely
many times, we obtain that $V\left( u_{i_2} \overrightarrow{P} v_{n-1}\right)\subseteq N_G(v_n).$ 
Therefore $V\left( u_{i_1} \overrightarrow{P} v_{n-1}\right)\subseteq N_G(v_n).$ 

Next we show that $V\left( u_s \overrightarrow{P} u_{i_1}\right)\subseteq N_G(v_n).$ 
 If $i_1=s,$ we are done. 
If $i_1\ge s+1$, then substituting  $a=s$ in Claim~\ref{special}(i) gives that 
            $u_{s}^+\sim u_{\ell-1}^{+2}=v_n.$  
            Hence $u_{s+1}=u_{s}^+$. 
            If $i_1=s+1,$ we are done. 
            If $i_1\ge s+2,$ then repeating the above argument  successively finitely
many times, we obtain that $V\left( u_{s} \overrightarrow{P} u_{i_1}\right)\subseteq N_G(v_n)$.  
By the above arguments, we have   
$V\left( u_{s} \overrightarrow{P} v_{n-1}\right)\subseteq N_G(v_n)$ 
 and thus   $V\left( u_{s} \overrightarrow{P} u_{i_1}\right)=A_2.$
Together with $|A_2|=t-1$, we have  $u_s=v_{n-t+1}$. This completes the proof of \eqref{et1}.

			We now return to the proof of Claim~\ref{inden}. 
            We first consider the case $s=t\ge 3$. Then $k'<k$. Since $G$ is non-hamiltonian, we have $k'\ne k-1$. Thus, $k'\le k-2.$ 
            By Claim~\ref{c0}, we have 
            $$|N_G(v_1)\cap [v_k,v_n]|=t-1\ge 2.$$ 
            Based on the fact that $G$ is non-hamiltonian and  \eqref{et1},  
            we have 
            $$N_G(v_1) \cap V\left( u_{s+1} \overrightarrow{P} v_{n}\right)=\emptyset.$$ 
            Then 
            \begin{align}\label{e4}
                |N_G(v_1) \cap[v_k,u_s^-]|\ge t-2\ge 1.
            \end{align}
            By Claim~\ref{special}(i) and \eqref{et1}, we have  
            \begin{align}\label{e3}
                E\left([u_s^+,v_n],[v_k,u_s^-]\right)=\emptyset.
            \end{align} 
            We assert that $v_{k+1}\not\sim v_{k'+1}.$ 
            Otherwise, $G$ contains a Hamilton cycle $v_1\overrightarrow{P}v_{k'}v_n\overleftarrow{P}v_{k+1}v_{k'+1}\overrightarrow{P}v_kv_1,$ a contradiction. 
            
            If $v_{k+1}\ne u_s$, then 
            $v_{k+1}\in [v_k,u_s^-]$. 
            Combining this with \eqref{e3} and $v_{k+1}\not\sim v_{k'+1},$ 
            we have 
             $$E \left([u_s^+,v_n]\cup\{v_{k'+1}\},B_1^-\cup \{v_{k+1}\} \right)=\emptyset,$$ which implies that there exists a $(t,t)$-bipartite-hole in $G$, 
             a contradiction.  
             
             Next we assume that $v_{k+1}=u_s.$ 
             Then $v_k=u_s^-$ and every inequality in \eqref{e4} becomes equality. Thus $t=s=3.$  
             By \eqref{et1}, we have $u_s=v_{n-2}$. Then $v_k=v_{n-3}$. 
             If $E\left(\{u_1^-,u_1^+,v_{k'+1}\}, \allowbreak \{v_{n-2},v_{n-1}\} \right)\neq\emptyset,$ then by Table~\ref{t1}, 
             $G$ contains a Hamilton cycle, a contradiction. 
             Hence $$E\left(\{u_1^-,u_1^+,v_{k'+1}\},\allowbreak \{v_{n-2},v_{n-1}\} \right)=\emptyset.$$ 
             Combining with the fact that 
             $N_G(v_n)=\{u_1,v_{k'},v_{n-2},v_{n-1}\}$ gives 
             $$E\left(\{u_1^-,u_1^+,v_{k'+1}\},\{v_{n-2},v_{n-1},v_n\} \right)=\emptyset,$$  which implies that $G$ contains a $(3,3)$-bipartite-hole, a contradiction. 

            \begin{table}[ht]
		\centering
		\caption{There exists an edge  $e\in E\left(\{u_1^-,u_1^+,v_{k'+1}\},\{v_{n-2},v_{n-1}\} \right)$}\label{t1}
		\begin{tabular}{@{}lll@{}} 
			\toprule 
			The edge $e$ &   
			A Hamilton cycle in $G$ \\ 
			\midrule 
			$e=u_1^- v_{n-2}:$ & $v_1 \overrightarrow{P}  u_1^- v_{n-2} 
			\overrightarrow{P} v_{n}  u_{1} \overrightarrow{P} v_{n-3} v_1$ \\
			$e=u_1^- v_{n-1}:$ & $v_1 \overrightarrow{P}  u_1^- v_{n-1} v_{n-2} v_n u_1 \overrightarrow{P} v_{n-3}   v_1$ \\
            $e=u_1^+ v_{n-2}:$ &  $v_1 \overrightarrow{P}  u_1 v_{n}\overleftarrow{P} v_{n-2}  u_1^+ \overrightarrow{P} v_{n-3}   v_1$ \\ 
            $e=u_1^+ v_{n-1}:$ &  $v_1 \overrightarrow{P}  u_1 v_{n} v_{n-2}  v_{n-1} u_1^+ \overrightarrow{P} v_{n-3}   v_1$ \\ 
            $e=v_{k'+1} v_{n-2}:$ &  $v_1 \overrightarrow{P}  v_{k'} v_{n} \overleftarrow{P} v_{n-2} v_{k'+1} \overrightarrow{P} v_{n-3}   v_1$ \\ 
            $e=v_{k'+1} v_{n-1}:$ &  $v_1 \overrightarrow{P}  v_{k'} v_{n}  v_{n-2} v_{n-1} v_{k'+1} \overrightarrow{P} v_{n-3}   v_1$ \\ 
			\bottomrule 
		\end{tabular}
	\end{table}

			Now we assume that $s\le t-1.$ 
            Then $k\le k'.$ 
            Together with  \eqref{et1}, 
			we have $n-t+1\le k'$.  
            Based on  
            $V\left( v_{n-t+1} \overrightarrow{P} v_{n-1}\right)\subseteq N_G(v_n)$ and  \eqref{34}, we have  $B_2=\emptyset$ and   $s=1$. Then by the definition of $k$ and $k'$, we have $k=2$ and $k'=n-1$. 
            By \eqref{et1} and the fact that  $G$ is non-hamiltonian, 
            $v_1$ is nonadjacent to any vertex of $V\left( v_{n-t+2} \overrightarrow{P} v_{n}\right).$ 
            Combining this with Claim~\ref{special}(i), 
            we have 
             \begin{align*}
                \text{
             $N_G (v_j)=N_G (v_j)\cap[v_2,v_n] = V\left( v_{n-t+1} \overrightarrow{P} v_{n}\right)\setminus\{v_j\},$ ~~for 
             $j\in\{n-t+2,\ldots, n\}.$   
             }
            \end{align*}
             Therefore, $v_{n-t+1} $ is a cut-vertex of $G$, which contradicts the assumption that $G$ is 2-connected. 
			This proves that  $A_2^+$ is an independent set of $G$. By symmetry,  $B_1^-$ is an  independent set of $G$. 
			This completes the proof of Claim~\ref{inden}.
		\end{proof} 
		We now return to the proof of Theorem~\ref{thmctOre} and proceed by considering the following two cases.
		\vskip 3mm
		\noindent{\bf Case 1.} $s\le t-1.$  
		\vskip 3mm
		
		In this case $k\le k'$.       
        If $t=2$, then $s=1$. By Claim~\ref{c0},  we have $d_G(v_1)=s+t-2=1,$ which contradicts the assumption that $G$ is $2$-connected. Next we assume that $t\ge 3.$ 
		We write
		\begin{align*}
			S_1 := (N_P(v_1)\cap[v_1,v_{k}])^-=\{w_1^-,\ldots,w_s^-\}. 
		\end{align*}
        Combining  with Claim \ref{inden} and \eqref{3.2} gives that  \begin{align}\label{epq}
                \text{
		$X:=S_1\cup A_2^+$ is  an independent set of $G$. 
        }\end{align}
        Together with the fact that  $G$ contains no $(s,t)$-bipartite-hole, we have 
    \begin{align}\label{epp}
                \text{
            $d_X(z)\ge t,$ for each vertex $z\in V(G)\setminus X$.
		}
            \end{align}  

        	\begin{claim}\label{middle} 
            \begin{wst}
			\item[{\rm (i)}] $u_s=v_k=w_s;$ 			
			\item[{\rm (ii)}] 
            If $s\le t-2,$ then
				 $\left|u_{i}^+ \overrightarrow{P} u_{i+1}^+\right|=3,$ for $s\le i\le t-2.$ 
            \end{wst}
		\end{claim} 
		\begin{proof}[Proof of Claim~\ref{middle}] 
            (i) By the definition of $k$, $w_s=v_k$. Suppose to the contrary that $|v_k\overrightarrow{P} u_s|\ge 2.$ Clearly, $v_k\notin X.$
            By Claim~\ref{special}(i), $v_{k}$ is nonadjacent to any vertex of $A_2^+.$ 
            Since $|A_2^+|=t-1$, we have $d_X(v_k)\le s\le t-1,$
            a contradiction to \eqref{epp}.

			(ii) By Claim~\ref{inden}, we have $\left|u_{i}^+ \overrightarrow{P} u_{i+1}^+\right|\ge 3,$ for  $s\le i\le t-2$. To the contrary,  suppose that there exists an integer $s\le j\le t-2$ such that  $\left|u_{j}^+ \overrightarrow{P} u_{j+1}^+\right|\ge  4.$ Clearly, $u_j^{+2}\notin X.$
			By Claim~\ref{special}(i), $u_{j}^{+2}$ is nonadjacent to any vertex of $F:=\{u_{j+1}^+,\dots,u_{\ell}^+\}.$ 
            Since $|F|\ge s$, 
            we have $d_X(u_{j}^{+2})\le t-1,$
            a contradiction to \eqref{epp}.  
            This proves Claim~\ref{middle}. 
		\end{proof} 

       By the definition of $k',$ $u_{t-1}=v_{k'}$. Based on Claim~\ref{middle}, $|S_1|=s,$ $|B_1|=t-1$ and the fact that $G$ is non-hamiltonian, 
        we have 
        \begin{align*}
            \text{$u_i=w_i,$~~for $s\le i\le t-1.$}
        \end{align*}
        Combining with Claim~\ref{special}, 
        we have that 
        \begin{align*}
            \text{$u_i$ is adjacent to each vertex of $S_1\cup A_2^+,$~~for $s\le i\le t-1.$}
        \end{align*} 
		
		\begin{claim}\label{right} 
        \begin{wst}
				\item[{\rm (i)}] 
		  $|u_{i}^{+}\overrightarrow{P}u_{i+1}^+|=3$, for $t-1\le i\le \ell-1$; 
          \item[{\rm (ii)}] 
			$|w_{j-1}^{-}\overrightarrow{P}w_{j}^-|=3,$ for $2\le j\le s.$ 
          \end{wst}
		\end{claim} 
		\begin{proof}[Proof of Claim~\ref{right}] 
        (i) By Claim~\ref{inden}, 
        $|u_{i}^{+}\overrightarrow{P}u_{i+1}^+|\ge 3$ for each integer $t-1\le i\le \ell-1$. 
Then we assert that   $|u_{i}^{+}\overrightarrow{P}u_{i+1}^+|\in\{3,4\},$~~for $t-1\le i\le \ell -1$.  
             Otherwise, suppose that  there exists an integer $t-1\le i\le \ell-1$ such that  $|u_{i}^{+}\overrightarrow{P}u_{i+1}^+|\ge 5.$  Clearly, $u_{i}^{+3}\not\in X.$
             By Claim~\ref{special}(i), $u_{i}^{+3}$ is non-adjacent to any vertex of $A_2^+.$ Since $|A_2^+|=t-1$, we have $d_X(u_i^{+3})\le s\le t-1,$
            a contradiction to \eqref{epp}.
        
If $t\ge 2s$, then we suppose to the contrary that there exists an integer $t-1\le j\le \ell-1$ such that $|u_{j}^{+}\overrightarrow{P}u_{j+1}^+|=4$. 
			 Clearly, $u_{j+1}\notin X.$
             By Claim~\ref{special}(i), we have $u_{j+1}$ is non-adjacent to any vertex of $F:=\{u_s^+,\dots,u_{j}^+\}.$ 
             Since $|F|=j-s+1\ge t-1-s+1\ge  s$,
            we have $d_X(u_{j+1})\le t-1,$
            a contradiction to \eqref{epp}.
        
        Next we assume that $t\le 2s-1$. 
We show that 
             \begin{align}\label{ey3}
                 \text{ 
                 $|u_{i}^{+}\overrightarrow{P}u_{i+1}^+|=3,$~~for $t-1\le i\le 2s-2$.  
                 }
             \end{align} Suppose to the contrary that there exists an integer $t-1\le j\le 2s-2$ such that $|u_{j}^{+}\overrightarrow{P}u_{j+1}^+|=4$. 
             
             We show that $|u_{t-1}^+\overrightarrow{P}u_t^+|=4.$  If $t=2s-1,$ then $j=t-1$ and $|u_{t-1}^{+}\overrightarrow{P}u_{t}^+|= 4$.  If $t\le 2s-2,$ then we will see that 
             \begin{align}\label{ey2}
                 \text{
             if $|u_{i}^{+}\overrightarrow{P}u_{i+1}^+|= 4$ 
             for some integer $t\le i\le 2s-2$,  then  $|u_{i-1}^{+}\overrightarrow{P}u_{i}^+|=4$. 
             }
             \end{align}
             Otherwise, suppose $|u_{i-1}^{+}\overrightarrow{P}u_{i}^+|=3$. Note that $u_{i+1}\notin X.$ By \eqref{epp}, $d_X(u_{i+1})\ge t$.
             Hence, there exist a vertex $w_j^-\in S_1\cap N_G(u_{i+1})$.  
             By Claim~\ref{special}(i), we have 
             $u_{i-1}^+\sim u_i^{+2}.$ Then  
             $G$ contains a Hamilton cycle $$v_1\overrightarrow{P}w_j^-u_{i+1}\overrightarrow{P}v_nu_i \overrightarrow{P} u_{i}^{+2}u_{i-1}^+\overleftarrow{P}w_jv_1,$$ a contradiction. 
             This proves \eqref{ey2}. 
			 Thus, by (\ref{ey2}), we have $|u_{t-1}^{+}\overrightarrow{P}u_{t}^+|= 4$.

             Since $G$ contains no $(s, t)$-bipartite-hole, we have  $E\left(S_1, A_2^+ \cup\{u_{t}\} \right)\neq \emptyset.$ Then  there exists a vertex $w_f^-\in S_1$ such that $w_f^-\sim u_t$. 
             If $s\le t-2$, then 
			 by Claim~\ref{middle}(ii), we have  $u_{t-2}^{+2}=u_{t-1}$. By Claim~\ref{special}(i), we have 
             $u_{t-2}^+\sim u_{t-1}^{+2}.$
			 Then $G$ contains a Hamilton cycle   $$v_1\overrightarrow{P}w_f^-u_{t}\overrightarrow{P}v_nu_{t-1}u_{t-1}^+u_{t-1}^{+2}u_{t-2}^+\overleftarrow{P}w_fv_1,$$ a contradiction.  
             Next we assume that  $s=t-1.$  By Claim~\ref{special}(i),  $u_{s}^{+2}$ is non-adjacent to any vertex of $A_2^+\setminus\{u_s^+\}.$ 
             Suppose that there exists a vertex $w_q^-\in S_1$, such that 
             $w_q^-\not\sim u_s^{+2}.$ 
             Then 
             $$E\left(\{w_q^-\}\cup \left( A_2^+\setminus\{u_s^+\} \right), \left( S_1\setminus\{w_q^-\} \right) \cup\{u_s^+,u_s^{+2}\} \right)=\emptyset.$$ 
             Hence, there exists an $(s,t)$-bipartite-hole in $G$, a contradiction. 
             This implies that  $S_1\subseteq N_G(u_{s}^{+2})$. Thus  $w_{s}^-\sim u_{s}^{+2}.$  Since $E\left(S_1, A_2^+ \cup\{u_{s+1}\} \right)\neq \emptyset,$ there exists a vertex $w_p^-\in S_1$ such that $w_p^-\sim u_{s+1}$. 
             By Claim~\ref{middle}(i), $u_s=w_s.$
             Then $G$ contains a Hamilton cycle $$v_1\overrightarrow{P}w_p^-u_{s+1}\overrightarrow{P}v_nu_{s}\overrightarrow{P} u_{s}^{+2}w_{s}^-\overleftarrow{P}w_pv_1,$$ a contradiction.  This proves \eqref{ey3}.  

             We now return to the proof of Claim~\ref{right}. 
             If $s=t-1$, that is, $\ell-1=2s-2$, then the desired  result  follows from  \eqref{ey3}. 
             Next we 
             assume that $s\le t-2.$ By \eqref{ey3}, it suffices to prove  that $|u_{i}^{+}\overrightarrow{P}u_{i+1}^+|=3$ for $2s-1\le i\le \ell-1$. Otherwise, there exists an integer  $2s-1\le j\le \ell-1$ such that $|u_{j}^{+}\overrightarrow{P}u_{j+1}^+|\ge 4$. Clearly, $u_{j+1}\notin X.$
			 By Claim~\ref{special}(i), we have $u_{j+1}$ is non-adjacent to any vertex of $F:=\{u_s^+,\dots,u_{j}^+\}.$ Since $|F|=j-s+1\ge s$,  we have $d_X(u_{j+1})\le t-1,$ a contradiction to \eqref{epp}.
             This proves (i).  
             
             (ii) It follows from (i) and the symmetry of $v_1$ and $v_n$. 
             This completes the proof of Claim~\ref{right}. 
		\end{proof}

		By Claims~\ref{middle} and \ref{right}, we have $|P|=2\ell+1$. 
        By \eqref{epq}, $X$ is an independent set of cardinality $\ell+1.$ Together with the condition that $\sigma_2(G)\ge 2\ell$, we have that each vertex of $X$ is adjacent to each vertex of $V(G)\setminus X.$ 
        Hence  $G=G_\ell\vee (\ell+1 )K_1,$ where $G_\ell$ is an arbitrary graph of order $\ell,$ as desired. 
		
		\vskip 3mm
		\noindent{\bf Case 2.} $s=t$. 
		\vskip 3mm
Since $s=t\ge 2$ and  $G$ is non-hamiltonian,  
			we have $k'<k$ and $k'\ne k-1$. 
			Thus $k'\le k-2$. By \eqref{eq4}, \eqref{34} and $|N(v_1)\cap [v_1,v_k]|=s,$  we have  
            \begin{align*}
                \text{
                $N(v_1)\cap [v_{k'+1},v_{k-1}]=\emptyset$~~and~~$N(v_n)\cap [v_{k'+1},v_{k-1}]=\emptyset.$
                }
            \end{align*} 
        Since $G$ is non-hamiltonian, we have 
			$$E \left( (N_P(v_1)\cap [v_1,v_k])^-,(N_P(v_{k'+1})\cap [v_{k+1},v_n])^- \right)=\emptyset,$$
			and 
			$$E \left( (N_P(v_n)\cap [v_{k'},v_n])^+,(N_P(v_{k-1})\cap [v_1,v_{k'-1}])^+ \right)=\emptyset.$$ 
			By \eqref{3.2}, \eqref{3.2'} and the fact that $G$ contains no $(s, t)$-bipartite-hole, we have 
            \begin{align}\label{ey4}
                (N_P(v_{k'+1})\cap [v_{k+1},v_n])^-\subseteq A_2^+ \text{~~and~~}
			(N_P(v_{k-1})\cap [v_{1},v_{k'-1}])^+\subseteq B_1^-.
            \end{align} 
			We assert that 
             \begin{align}\label{ey5}
                \text{
            $N_G(v_{k-1})\cap B_1^-=\emptyset$~~and~~$N_G(v_{k'+1})\cap A_2^+=\emptyset$. 
            }
            \end{align} 
            Suppose to the contrary that  there exists  a vertex  $x\in N_G(v_{k-1})\cap B_1^-.$ 
			By \eqref{ey4}, we have  $x^+\in B_1^-,$ which contradicts Claim~\ref{inden}. 
            This implies that $N_G(v_{k-1})\cap B_1^-=\emptyset$. 
            Similar reasoning shows that 
            $N_G(v_{k'+1})\cap A_2^+=\emptyset.$ Thus \eqref{ey5} holds. 
		\begin{claim}\label{k'}
			$k'=k-2.$
		\end{claim}
		\begin{proof}[Proof of Claim~\ref{k'}] 
			To the contrary, suppose that $k'\le k-3.$ 
            Then $v_{k'+1}\neq v_{k-1}.$ 
			We show that 
            \begin{align}\label{ey7}
                \text{
            $N_G(v_{k'+1})\cap B_1^-=\emptyset$~~and~~ 
            $N_G(v_{k-1})\cap A_2^+=\emptyset.$ 
            }
            \end{align} 
            In fact, for any vertex $x\in B_1^-$ there exists a Hamilton $(v_{k'+1},x)$-path  $v_{k'+1}\overrightarrow{P}v_nv_{k'}\overleftarrow{P}x^+v_1\overrightarrow{P}x,$ 
            we have $x\not\sim v_{k'+1}$ and thus $N_G(v_{k'+1})\cap B_1^-=\emptyset$. Similar reasoning shows that $N_G(v_{k-1})\cap A_2^+=\emptyset.$ Thus \eqref{ey7} holds. 
			Choose a vertex $x_0\in B_1^-$. Combining with Claim~\ref{inden}, \eqref{3.2}, \eqref{ey5} and \eqref{ey7}, we have  
			$$E\left( (B_1^-\setminus\{x_0\})\cup\{v_{k'+1},v_{k-1}\},A_2^+\cup \{x_0\} \right)=\emptyset,$$
			which implies there exists an $(s,s)$-bipartite-hole in $G$, a contradiction. 
            This proves Claim~\ref{k'}. 
		\end{proof} 
Combining with \eqref{3.2}, \eqref{ey5}, Claims~\ref{inden} and \ref{k'}, 
            we have that 
            \begin{align}\label{ep6}
                \text{
                $v_{k'+1}=w_s^-=u_{s-1}^+,$ and  
            $Y:=B_1^-\cup \{v_{k'+1}\}\cup A_2^+$ is an  independent set of $G$.
		}
            \end{align}   
     Together with the fact that  $G$ contains no $(s,s)$-bipartite-hole, we have 
    \begin{align}\label{ep7}
                \text{
            $d_Y(z)\ge s,$ for each vertex $z\in V(G)\setminus Y$.
		}
            \end{align}  
            
We first consider the case that  $s=t=2$. Suppose that $k'\ge 3.$ Since $v_{k'-1}\overleftarrow{P}v_1v_k\overrightarrow{P}v_nv_{k'}v_{k'+1}$ is a Hamilton $(v_{k'-1}, v_{k'+1})$-path, we have  $v_{k'-1}\not\sim v_{k'+1}$. 
        By Claim~\ref{c0} and \eqref{eq4}, we have $d_G(v_1)=d_G(v_n)=2.$ Thus $N_G(v_1)=\{v_2,v_k\}$ and 
        $N_G(v_n)=\{v_{k'},v_{n-1}\}.$
        Then $v_1 \not\sim v_{k'+1}$ and $v_n\not\sim v_{k'-1}$. 
		Together with \eqref{ep6}, we have 
        $$E\left(\{v_1,v_{k'-1}\},\{v_{k'+1},v_n\}\right)=\emptyset,$$ 
        which implies that  $G$ contains a $(2,2)$-bipartite-hole, a contradiction. 
        This implies that $k'=2.$ 
        Similar reasoning shows that $k=n-1.$ 
 Then $n=5$. Combining with \eqref{ep6}, we have $G= G_2\vee 3K_1$, where $G_2$ is an arbitrary graph of order two, as desired.  
 
 Next we assume that $s=t\ge 3.$	         
        By \eqref{ep6}, we have  $v_{k'+1}\not\sim v_n.$
Then $$d_G(v_{k'+1})+d_G(v_n)\ge 2\widetilde{\alpha}(G)-2=4(s -1).$$ Note that there exists a Hamilton $(v_1,v_{k'+1})$-path 
$v_1 \overrightarrow{P}v_{k'} v_n \overleftarrow{P} v_{k'+1}.$   
By the choice of $P$, we have $d_G(v_{k'+1})\le d_G(v_n).$ Together with \eqref{eq4}, we have  
\begin{align*}
d_G(v_{k'+1})=2s-2.
\end{align*}
Combining this with \eqref{ey4}, we have 
            \begin{align}
                N_P(v_{k'+1})\cap [v_{k+1},v_n]&= (A_2\setminus\{u_{\ell}\})^{+2}=\{u_s^{+2}, \ldots, u_{\ell-1}^{+2}\} 
            \text{~~and~~}\notag\\
                N_P(v_{k'+1})\cap [v_{1},v_{k'-1}]&= (B_1\setminus\{w_1\})^{-2}=\{w_2^{-2}, \ldots,w_{t-1}^{-2}\}.\label{l}
            \end{align}

		\begin{claim}\label{1}
			\begin{wst}
				\item[{\rm (i)}]
				$|w_i^- \overrightarrow{P} w_{i+1}^-|=3,$ for $1\le i\le s-1$; 
				\item[{\rm (ii)}]    
				$|u_j^+ \overrightarrow{P} u_{j+1}^+|=3,$ for $s-1\le j\le \ell-1$. 
			\end{wst}
		\end{claim}
		\begin{proof}[Proof of Claim \ref{1}]
             We first prove that 
             \begin{align}\label{ep2}
                \text{
             $|w_i^- \overrightarrow{P} w_{i+1}^-|\le 4,$~~for $1\le i\le s-1.$ 
             }
        \end{align}
             Otherwise, 
             suppose to the contrary that there exists an integer $1\le i\le s-1$ such that $|w_i^- \overrightarrow{P} w_{i+1}^-|\ge 5$. Clearly, $w_i^+\notin Y.$
             By Claim~\ref{special}(ii), $w_i^+$ is non-adjacent to any vertex of $B_1^-$. 
             By \eqref{l}, $w_i^+\not\in  N_P(v_{k'+1}).$ Then $d_Y(w_i^+)\le s-1,$ a contradiction to \eqref{ep7}.
            This implies that 
            $|w_i^- \overrightarrow{P} w_{i+1}^-|\le 4$ for $1\le i\le s-1.$ Similar reasoning shows that 
            \begin{align}\label{ee1}
                \text{
                $|u_j^+ \overrightarrow{P} u_{j+1}^+|\le 4,$~~for $s-1\le j\le \ell-1$. 
                }
            \end{align} 
            We assert that 
            \begin{align}\label{w_i}
                \text{if $|w_i^- \overrightarrow{P} w_{i+1}^-|= 4$ for some integer $1\le i\le s-2$, then $|w_{i+1}^- \overrightarrow{P} w_{i+2}^-|= 4$.}
            \end{align} 
            Otherwise, suppose  $|w_{i+1}^- \overrightarrow{P} w_{i+2}^-|\neq 4.$ 
            Then by \eqref{ep6}, we have $|w_{i+1}^- \overrightarrow{P} w_{i+2}^-|= 3.$ 
            By Claim~\ref{special}(ii), we have  $w_{i+2}^{-}\sim w_{i+1}^{-2}.$ By \eqref{l}, $w_i\not\sim v_{k'+1}$. Note that $w_i\notin Y$. By \eqref{ep7}, we have $d_Y(w_i)\ge s$. Hence, there exists a vertex $u_p\in  A_2$ such that $w_i\sim u_p^+$. Then  $G$ contains a Hamilton cycle $$v_1w_{i+1}\overleftarrow{P}w_{i+1}^{-2}w_{i+2}^{-}\overrightarrow{P}u_pv_n\overleftarrow{P}u_p^+w_i\overleftarrow{P}v_1,$$ a contradiction. This proves \eqref{w_i}. 
            Similar reasoning shows that 
            \begin{align}\label{u_j}
                \text{if $|u_j^+ \overrightarrow{P} u_{j+1}^+|= 4$ for some $s\le j\le \ell-1$, then $|u_{j-1}^+ \overrightarrow{P} u_{j}^+|= 4.$ 
                }
            \end{align} 
            
We are now prepared to prove (i). 
                
            (i) 
            By \eqref{ep6} and \eqref{ep2}, we have 
        $ |w_i^- \overrightarrow{P} w_{i+1}^-|\in\{3, 4\}$  for $1\le i\le s-1.$  
            To the contrary, suppose that there exists an integer $1\le q\le s-1$
            such that 
            $|w_q^- \overrightarrow{P} w_{q+1}^-|= 4$. By \eqref{w_i}, we have $|w_{s-1}^- \overrightarrow{P} w_{s}^-|= 4$.
            
            We first assume that  $|u_{s-1}^+\overrightarrow{P}u_s^+|=3$. By \eqref{ep6}, \eqref{ee1} and  \eqref{u_j},  we have 
            \begin{align}\label{ew1}
                \text{
            $|u_j^+ \overrightarrow{P} u_{j+1}^+|=3$ for $s-1\le j\le \ell-1$, 
            and hence $|v_k\overrightarrow{P}v_n|=2s-2$.
            }
            \end{align}
            By \eqref{l}, we have $w_{s-1}\not\sim v_{k'+1}$. Note that $w_{s-1}\notin Y$. By \eqref{ep7}, we have $d_Y(w_{s-1})\ge s.$ 
            This implies that there exists a vertex $u_j\in A_2$ such that $w_{s-1}\sim u_j^+$.  
            By Claim~\ref{special}(ii), we have $d_G(w_{s-1}^-)=2s-2$ and $|N_G(w_{s-1}^-)\cap[v_1,v_{k'}]|=s-1$.  
           Then $|N_G(w_{s-1}^-)\cap[v_k,v_n]|=s-1$.  
            Combining this with \eqref{ew1} and \eqref{ep6}, $w_{s-1}^-$ is adjacent to every vertex of $V(v_k\overrightarrow{P}v_n)\setminus A_2^+$. Hence, $u_j\sim w_{s-1}^-.$ By \eqref{ep6} and $|w_{s-1}^- \overrightarrow{P} w_{s}^-|= 4$, we have $w_{s-1}=v_{k'-1}$. Then $G$ contains a Hamilton cycle 
            $$v_1\overrightarrow{P}w_{s-1}^-u_j\overleftarrow{P}v_{k'}v_n\overleftarrow{P}u_j^+w_{s-1}v_1,$$
			a contradiction.
            
            Next we assume that $|u_{s-1}^+\overrightarrow{P}u_s^+|=4$. Then by Claim~\ref{k'}, $u_s\neq  v_k.$  We show that 
            \begin{align}\label{ea1}
                \text{
            $|w_i^- \overrightarrow{P} w_{i+1}^-|= 4,$ for each integer $1\le i\le s-1$.
            }
            \end{align} 
            Recall that there exists an integer $1\le q\le s-1$
            such that 
            $|w_q^- \overrightarrow{P} w_{q+1}^-|= 4$. 
            If $q=1$, then \eqref{ea1} follows from \eqref{w_i}. Next we assume that $q\ge 2$.  We assert  that   $|w_{q-1}^- \overrightarrow{P} w_{q}^-|= 4$. 
                 Suppose to the contrary that $|w_{q-1}^- \overrightarrow{P} w_{q}^-|= 3.$ 
            Note that $u_s\neq  v_k.$ From Claim~\ref{special}(i),  $v_k$ is non-adjacent to any vertex of $A_2^+.$ 
            Together with \eqref{ep7}, we have 
             $B_1^-\subseteq N_G(v_k)$.  
            Then $v_k\sim w_{q-1}^-.$ 
            By  $|w_{q-1}^- \overrightarrow{P} w_{q}^-|= 3$  and \eqref{l}, 
            we have $v_{k'+1}\sim w_q^{-2}=w_{q-1}.$ By \eqref{ep6} and $|w_{s-1}^- \overrightarrow{P} w_{s}^-|= 4$, we have $w_{s-1}=v_{k'-1}$. Thus  $G$ contains a Hamilton cycle
            $$v_1\overrightarrow{P}w_{q-1}^-v_k\overrightarrow{P}v_nv_{k'}v_{k'+1}w_{q-1}\overrightarrow{P}w_{s-1}v_1,$$
            a contradiction. 
            This implies that  
            $|w_{q-1}^- \overrightarrow{P} w_{q}^-|= 4$. 
            Now repeating the above argument successively finitely many times and combining with \eqref{w_i}, we proves 
            \eqref{ea1}. 
            
            For each $1\le f\le s-1,$ 
            by \eqref{ea1}, \eqref{ep6} and \eqref{l}, we have $v_{k'+1}=w_s^-\not\sim w_f.$ 
            Combining this with Claim~\ref{special}(ii), 
             $w_f$ is non-adjacent to any vertex of $\{w_{f+1}^-,\dots,w_s^-\}$. By \eqref{ep7}, we have $d_{A_2^+}(w_{f})\ge  s-f$. 
            Similarly, 
            by \eqref{ea1}, $w_1^-=v_1\not\sim w_{f+1}^{-2}.$
            Combining this with Claim~\ref{special}(ii), we obtain that 
             $w_{f+1}^{-2}$ is non-adjacent to any vertex of $\{w_{1}^-,\dots,w_{f}^-\}$. 
            Then by \eqref{ep7}, $d_{A_2^+}(w_{f+1}^{-2})\ge f$.  
            Let $T=V(v_1\overrightarrow{P}v_k')\setminus B_1^-.$ 
             Then  $$e(T,A_2^+)\ge (s-1)(s-f+f)=s(s-1).$$ In the other hand, by Claim~\ref{special}(i) and \eqref{ep6}, for any vertex $u_a\in A_2$, we have $|N_G(u_a^+)\cap [v_1,v_{k'}]|=s-1$, and thus  $$e(T,A_2^+)= (s-1)(s-1),$$ a contradiction. 
            This proves (i).

            (ii) By the symmetry of $v_1$ and $v_n$, the result holds, and we omit the proof.
	\end{proof}
By Claims~\ref{k'} and \ref{1}, $|P|=2\ell +1.$ 
By \eqref{ep6}, we have  $Y$ is an  independent set of cardinality $\ell+1.$ Together with the condition that $\sigma_2(G)\ge 2\ell$, we obtain that  each vertex of $Y$ is adjacent to each vertex of $V(G)\setminus Y.$ 
        Hence  $G=G_\ell\vee (\ell+1 )K_1,$ where $G_\ell$ is an arbitrary graph of order $\ell.$ 
This completes the proof of Theorem~\ref{thmctOre}. 
\end{proof}
    \begin{proof}[\bf Proof of Theorem~\ref{thmct1}] 
    Let $G$ be a non-hamiltonian graph satisfying the conditions of Theorem~\ref{thmct1}. 
    By Lemma~\ref{lemZ}, $G$ is connected. 
    If $G$ is $2$-connected, then by Theorem~\ref{thmctOre}, we have $G= G_{ \frac{n-1}{2}}\vee \frac{n+1}{2} K_1$, as desired. Next we assume that $G$ contains a cut vertex $x$.  
    By  definition, there exist two 
		positive integers $s$ and $t$ such that $s+t=\widetilde{\alpha}(G) +1$ 
		and $G$ contains no $(s,t)$-bipartite-hole.  
		Without loss of generality, we assume that $1\le s\le t.$ 
		Since $G$ has no Hamilton cycle, $G\neq K_n$, and hence $t\ge 2$ and $\widetilde{\alpha}(G) \ge 2.$ 
    We let $H_1,\ldots,H_k$ be the components of $G-x$, where $k\ge 2$ and   $|H_1|\le  \cdots\le |H_k|.$ 
		If $|H_1|\le s-1$, then by $\delta(G)\le |H_1|$, we have 
		$$s+t-1= \widetilde{\alpha}(G)\le \delta(G)+1\le |H_1|+1\le s.$$
		Then $t\le 1,$ a contradiction. 
        This implies that 
        $|H_1|\ge s.$ 
		Since $G$ contains no $(s,t)$-bipartite-hole
        and 
        $E\left( V(H_1),  \cup_{i=2}^k V(H_i) \right) =\emptyset,$ 
        we have  
		$ n-1-|H_1|=\sum_{i=2}^k |H_i|\le t-1.$ 
		Then 
		\begin{align}\label{e1}
        n-1-|H_1|\le t-1\le s+t-2=\widetilde{\alpha}(G)-1\le \delta(G)\le |H_1|. 
		\end{align}
		Thus, $|H_1|\ge (n-1)/2$. 
		By  the fact that $|H_1|=\min\{|H_j|: 1\le j\le k\} $  and  $\sum_{j=1}^k |H_j|=n-1$,  
		we have $k=2$, $|H_1|=|H_2|=(n-1)/2,$ where $n$ is odd.  Then every inequality in \eqref{e1} becomes equality. Hence $\widetilde{\alpha}(G)=(n+1)/2.$  
		
		\begin{fact}\label{c1}
			$x$ is adjacent to all vertices of $V(G)\setminus  \{x\}.$
		\end{fact}
		\begin{proof}[Proof of Fact~\ref{c1}]
			Without loss of generality, we suppose  to the contrary that $x\not\sim v_1$, where $v_1\in V(H_1).$ Since 
			$E(v_1,V(H_2)\cup\{x\})=\emptyset$, there exists a 
			$(1,\frac{n+1}{2})$-bipartite-hole in $G$. Combining with   $|H_1|=|H_2|=(n-1)/2$ gives that $G$ contains an  
			$(a,b)$-bipartite-hole for any pair of nonnegative integers $a,b$  with $a+b=(n+3)/2$, which contradicts the fact that  $\widetilde{\alpha}(G)=(n+1)/2.$
		\end{proof}
		
		\begin{fact}\label{c2}
			$H_i$ is a complete graph, for $i\in\{1,2\}.$
		\end{fact}
		\begin{proof}[Proof of Fact~\ref{c2}]
			Without loss of generality, we suppose  to the contrary that  $v_1\not\sim v_2$ where $v_1,v_2\in V(H_1)$. 
			Note that $E(v_1,V(H_2)\cup\{v_2\})=\emptyset.$ 
			There exists a $(1,\frac{n+1}{2})$-bipartite-hole in $G$. 
			Since $|H_1|=|H_2|=(n-1)/2$, $G$ contains an $(a,b)$-bipartite-hole for any pair of nonnegative integers $a,b$  with  $a+b=(n+3)/2$, a contradiction.
		\end{proof}
		
		By Facts~\ref{c1} and \ref{c2}, we have $G=K_1\vee 2K_{\frac{n-1}{2}},$ where $n$ is odd. 
        This completes the proof of Theorem~\ref{thmct1}.
    \end{proof}
	\section{\normalsize Concluding remarks} 
	In this paper, we first 
    characterize all $2$-connected non-hamiltonian graphs $G$ satisfying $\sigma_2(G)\ge 2\widetilde{\alpha}(G)-2.$ Then we obtain a stability result of the McDiarmid-Yolov theorem. 
	A graph $G$ with $\delta(G)\ge (n-1)/2$ contains no $(1, \lfloor (n+1)/2\rfloor)$-bipartite-hole, hence 
	$\delta(G) \ge (n-1)/2 \ge \widetilde{\alpha}(G)-1.$ 
	Then the subsequent  known result follows 
	immediately from Theorem~\ref{thmct1}. 
	\begin{cor}
		Let $G$ be a graph of order $n\ge 3$.  
		If $\delta(G)\ge (n-1)/2$, 
		then $G$ is hamiltonian unless $n$ is odd and 
		$G\in\{ G_{ \frac{n-1}{2}}\vee \frac{n+1}{2} K_1, K_1\vee 2K_{ \frac{n-1}{2}} \}$, where $G_{ \frac{n-1}{2}}$ is an arbitrary graph of order $(n-1)/2$. 
	\end{cor}

    By joining a new vertex to a graph $G$, 
	we obtain  the following two results from Theorems~\ref{thmctOre} and \ref{thmct1}, respectively.
    
	\begin{cor} 
		Let $G$ be a connected graph.  
		If $\sigma_2(G)\ge 2\widetilde{\alpha}(G)-4$, 
		then $G$ is traceable 
		unless 
		$G= G_{ \frac{n-2}{2}}\vee \frac{n+2}{2} K_1$, where $n\ge 4$ is even and  $G_{ \frac{n-2}{2}}$ is an arbitrary graph of order $(n-2)/2$.
	\end{cor}

	\begin{cor} 
		Let $G$ be a graph.  
		If $\delta(G)\ge \widetilde{\alpha}(G)-2$, 
		then $G$ is traceable 
		unless 
		$G\in\{ G_{ \frac{n-2}{2}}\vee \frac{n+2}{2} K_1, 
		2K_{ \frac{n}{2}} \}$, where $n$ is even and  $G_{ \frac{n-2}{2}}$ is an arbitrary graph of order $(n-2)/2$.
	\end{cor} 

    \section*{\normalsize Acknowledgement}  The authors are grateful to Professor Xingzhi Zhan for his constant support and guidance. This research  was supported by the NSFC grant 12271170.

\section*{\normalsize Declaration}

\noindent\textbf{Conflict~of~interest}
The authors declare that they have no known competing financial interests or personal relationships that could have appeared to influence the work reported in this paper.

\noindent\textbf{Data~availability}
No data was used for the research described in the article.

\end{document}